\documentclass[10pt,reqno]{amsart}
\usepackage{amssymb}
\usepackage[all]{xy}

\theoremstyle{plain}
\newtheorem{prop}{Proposition}
\newtheorem{theo}[prop]{Theorem}
\newtheorem{coro}[prop]{Corollary}
\newtheorem{lemm}[prop]{Lemma}
\theoremstyle{remark}

\theoremstyle{definition}

\numberwithin{equation}{section}

\newcommand{\A}{{\mathbb A}}

\newcommand{\PP}{{\mathbb P}}

\newcommand{\G}{{\mathbb G}}
\newcommand{\N}{{\mathbb N}}

\newcommand{\Z}{{\mathbb Z}}

\newcommand{\eqto}{\stackrel{\lower1.5pt\hbox{$\scriptstyle\sim\,$}}\to}

\DeclareMathOperator{\Pic}{Pic}
\DeclareMathOperator{\Spec}{Spec}

\DeclareMathOperator{\Hom}{Hom}

\DeclareMathOperator{\Br}{Br}

\begin{document}
\title[Stable rationality]{Stable rationality and conic bundles}
\author{Brendan Hassett}
\address{
  Department of Mathematics, MS 136,
  Rice University,
  Houston, TX 77005, USA
}
\email{hassett@rice.edu}
\author{Andrew Kresch}
\address{
  Institut f\"ur Mathematik,
  Universit\"at Z\"urich,
  Winterthurerstrasse 190,
  CH-8057 Z\"urich, Switzerland
}
\email{andrew.kresch@math.uzh.ch}
\author{Yuri Tschinkel}
\address{
  Courant Institute,
  251 Mercer Street,
  New York, NY 10012, USA
and
  Simons Foundation,
  160 Fifth Avenue,
  New York, NY 10010, USA
}
\email{tschinkel@cims.nyu.edu}

\date{March 28, 2015}

\begin{abstract}
We study stable rationality properties of conic bundles over
rational surfaces.
\end{abstract}

\maketitle

\section{Introduction}
\label{sec.intro}
Let $k$ be an algebraically closed field, $X$ a smooth
projective variety over $k$.
We say $X$ is \emph{stably rational} if $X\times \PP^n$ is rational for
some $n\in \N$.
A breakthrough in the problem of addressing the stable rationality of
varieties was achieved by Voisin \cite{voisin}, who gave a remarkably
powerful sufficient criterion for the very general member of a
family of varieties to fail to be stably rational.
Working over the complex numbers,
she shows that the double cover of $\PP^3$ branched along
a very general quartic surface is not stably rational.
The degeneration aspects of Voisin's method were generalized
and simplified by Colliot-Th\'el\`ene and Pirutka \cite{colliotthelenepirutka}. 
These techniques have been applied to other families of double covers by Beauville
\cite{beauvillefourfive,beauvillesextic}; using 
a mixed characteristic version, introduced in \cite{colliotthelenepirutka}, 
and an idea of Koll\'ar \cite{kollar}, Totaro proved that many hypersurfaces in $\PP^n$ 
are not stably rational \cite{totaro}.

The problem of rationality of conic bundles has a long history; see, e.g.,
\cite{beauvilleprym} \cite[Conjecture I]{iskovskikh} \cite{sarkisovbir}.
See also \cite[Main Thm.]{shokurov} for rationality criteria.

An obvious necessary condition for (stable) rationality is
unirationality.
We recall, $X$ is unirational if there exists a dominant rational map from
a projective space to $X$.
Some cases of conic bundles, known to be unirational, are identified in
\cite{mella}, e.g., those with smooth branch locus in $\PP^2$
of degree $\le 8$.

In this paper we address the question of stable rationality of
conic bundles over rational surfaces.
Our main result is:

\begin{theo}
\label{thm.main}
Let $S$ be a rational smooth projective surface over an
uncountable algebraically closed field $k$ of characteristic different from $2$.
Let $\mathcal{L}$ be a linear system of effective divisors on $S$
whose general member is smooth and irreducible.
Let $\mathcal{M}$ an irreducible component of
the space of reduced nodal curves in $\mathcal{L}$
together with degree $2$ \'etale covering.
Assume that $\mathcal{M}$ contains
a cover, nontrivial over every irreducible component of a reducible curve
with smooth irreducible components.
Then the conic bundle over $S$ corresponding
to a very general point of $\mathcal{M}$ is not stably rational.
\end{theo}

By \cite[Thm.~1]{AM}, for each point $\{D'\rightarrow D\}\in \mathcal{M}$
there is a conic bundle over $S$ with discriminant curve $D$; this is 
defined up to birational isomorphism.
We remark that covers acquiring ramification above the nodes
are excluded from $\mathcal{M}$.

For example, Theorem \ref{thm.main} is applicable to the complete linear system
of degree $d$ curves in $\PP^2$ for $d\ge 6$.
A similar analysis is applicable, e.g., to Hirzebruch surfaces.
(See \cite{Beau,HKT}
for analysis of the monodromy of two-torsion points of Jacobians of curves in
rational ruled surfaces.) Our method is not applicable when the discriminant is a plane quintic
or a tri-section of a rational ruled surface; reducible curves 
of these types necessarily have rational components.

\medskip
\noindent
\textbf{Acknowlegdments:}
The first author was supported by NSF grants 1148609 and 1401764.
The second author was supported by the Swiss National Science Foundation.
The third author was supported by NSF grant 1160859.
We are grateful to I.~Cheltsov, J.-L.~Colliot-Th\'el\`ene, 
L.~Katzarkov, A.~Pirutka, and B.~Totaro for comments and suggestions.

\section{Background}
\label{sec.background}

\subsection{Brauer group}
\label{ss.brauergroup}
The Brauer group, for us, is the cohomological Brauer group
\[ \Br(S):=H^2(S,\G_m)_{\mathrm{tors}} \]
of a Noetherian scheme or Deligne-Mumford stack $S$.
(Here and elsewhere we work with \'etale cohomology groups.)
The role of points and residue fields of a scheme is played by
\emph{residual gerbes}
\begin{equation}
\label{eqn.residualgerbe}
\mathcal{G}_\xi\to \Spec(k(\xi))
\end{equation}
at points $\xi$ of a Noetherian Deligne-Mumford stack $S$
(cf.\ \cite[\S 11]{LMB}, \cite[App.\ B]{rydhjalg}).
In particular, if $S$ is integral, there is a residual gerbe $\mathcal{G}$
at the generic point of $S$.

\begin{prop}
\label{prop.Brbasic}
Let $S$ be a regular integral Noetherian Deligne-Mumford stack with residual gerbe
$\mathcal{G}$ at the generic point.
Then the restriction map $\Br(S)\to \Br(\mathcal{G})$ is injective.
Furthermore, if $\dim S=2$ then for any positive integer $n$, invertible in
the local rings of an \'etale atlas of $S$, the residue maps fit into
an exact sequence
\[
0\to \Br(S)[n]\to \Br(\mathcal{G})[n]\to
\bigoplus_{\xi\in S^{(1)}} H^1(\mathcal{G}_\xi,\Z/n\Z),
\]
where $S^{(1)}$ denotes the set of codimension $1$ points of $S$.
\end{prop}

\begin{proof}
This is standard for schemes; see, e.g., \cite{GB}.
The proof is valid for Deligne-Mumford stacks given the
vanishing of $H^1(\mathcal{G}_\xi,\Z)$, which
in turn follows from the standard
vanishing of $H^1(\Spec(k(\xi)),\Z)$ and the Leray spectral sequence for
\eqref{eqn.residualgerbe}.
\end{proof}

\subsection{Gerbes}
\label{ss.gerbes}
Let $X$ be a Noetherian Deligne-Mumford stack.
By \emph{gerbe} on $X$
we mean gerbe banded by a commutative algebraic group, always
$\G_m$ or $\mu_n$ with $n$ invertible in the local rings of
an \'etale atlas of $X$.
So (c.f.\ \cite[\S 4.2]{milne}),
a gerbe is a morphism of algebraic stacks $G\to X$ which \'etale locally
admits sections, \'etale locally lets any pair of sections be identified,
and which is equipped with compatible identifications of the automorphism
groups of local sections with the invertible regular functions, respectively
$\mu_n$-valued functions.
In the $\mu_n$ case, $G$ is also a Deligne-Mumford stack.

A gerbe is classified uniquely up to isomorphism, compatible with the
identification of automorphism groups of local objects, by a cohomology class
in
\[ H^2(X,\mu_n)\qquad\text{or}\qquad H^2(X,\G_m) \]
in the respective cases.

\begin{prop}
\label{prop.azumayagerbe}
Let $X$ be a Noetherian Deligne-Mumford stack and $n$ a positive integer
invertible in the local rings of an
\'etale atlas of $X$.
Then the boundary map of nonabelian cohomology for the central exact sequence
of groups
\[ 1\to \mu_n\to SL_n\to PGL_n\to 1 \]
associates to
a sheaf of Azumaya algebras $\mathcal{A}$ on $X$ of index $n$
a $\mu_n$-gerbe
\[ G\to X \]
and a locally free sheaf $E$ of rank $n$
on $G$ such that $E^\vee\otimes E$ is identified with
the pullback of $\mathcal{A}$.
\end{prop}

\begin{proof}
The class $\gamma\in H^2(X,\mu_n)$, the image of the class of
$\mathcal{A}$ by the boundary map, is the class of the
$\mu_n$-gerbe
\[ [P/SL_n], \]
where $P\to X$ denote the $PGL_n$-torsor associated with $\mathcal{A}$
(cf.\ \cite[\S I.2]{GB}).
The representation given by the inclusion $SL_n\to GL_n$ determines a
vector bundle $E$ on $[P/SL_n]$ and an isomorphism of
$E^\vee\otimes E$ with the pullback of $\mathcal{A}$.
\end{proof}

\begin{prop}
\label{prop.azumayatrivial}
In the situation of Proposition \ref{prop.azumayagerbe}, suppose that
$\mathcal{A}$ is trivial in the Brauer group.
Then there exists a line bundle $T$ on the $\mu_n$-gerbe $G$ such that
the $\mu_n$-action on fibers is via the inclusion $\mu_n\to \G_m$.
\end{prop}

\begin{proof}
By hypothesis, $\mathcal{A}\cong F^\vee\otimes F$
for some locally free sheaf $F$ of rank $n$ on $X$.
The $PGL_n$-torsor $P\to X$ associated with $\mathcal{A}$ may then be described
as the stack of
projective trivializations $\PP(\mathcal{O}_X^n)\cong \PP(F)$,
and the $\mu_n$-gerbe $G$ is $[P/SL_n]$.
The stack of compatible linear maps $\mathcal{O}_X^n\to F$ descends to a
line bundle on $G$ with the desired property.
\end{proof}

\subsection{Root stack}
\label{ss.rootstack}
Let $S$ be a regular integral Noetherian Deligne-Mumford stack,
$D\subset S$ an effective divisor, and $n$ a positive integer.
The corresponding \emph{root stack} will be denoted by
$\sqrt[n]{(S,D)}$ (\cite[\S 2]{cadman}, \cite[App.\ B]{AGV}).
This is a Deligne-Mumford stack when
$n$ is invertible in the local rings of an \'etale atlas of $S$.
Above $D$ is
the \emph{gerbe of the root stack} \cite[Def.\ 2.4.4]{cadman}, an
effective divisor of $\sqrt[n]{(S,D)}$ which is a $\mu_n$-gerbe over $D$.

\begin{prop}
\label{prop.ramificationtimesn}
Let $S$ be a regular integral Noetherian Deligne-Mumford stack of dimension $2$
and $D\subset S$ a regular closed substack of codimension $1$.
Then, for a positive integer $n$ invertible in the local rings of an
\'etale atlas of $S$ and $n$-torsion Brauer class $\alpha\in \Br(S\smallsetminus D)$
there exists a unique class
\[ \beta\in \Br\big(\sqrt{(S,D)}\,\big) \]
such that the restriction of $\beta$ to $S\smallsetminus D$ is equal to $\alpha$.
\end{prop}

\begin{proof}
The residue maps of $D$ and the gerbe of the root stack differ by
a factor of $n$, so we have the result by the exact sequence of
Proposition \ref{prop.Brbasic} applied to $\sqrt{(S,D)}$.
\end{proof}

\subsection{Algebraic cycles}
\label{ss.algebraiccycles}
Most of the material here can be found in \cite{ACTP}.
We recall that the triviality of the Brauer group of a proper smooth variety
is connected with the triviality of the Chow group $CH_0$
of $0$-cycles modulo rational equivalence upon extension of the base field.

\begin{prop}
\label{prop.chowdiagonal}
Let $k$ be an algebraically closed field, and
$X$ a proper smooth algebraic variety over $k$.
We let $\Delta\subset X\times X$ denote the diagonal and
\[ \mathrm{pr}_1,\, \mathrm{pr}_2\colon X\times X\to X \]
the projections.
For a field $L$ containing $k$ we let $X_L$ denote
$X\times_{\Spec(k)}\Spec(L)$.
The following conditions are equivalent:
\begin{itemize}
\item[(i)] The degree map $CH_0(X_L)\to \Z$ is an isomorphism for every
field extension $L/k$.
\item[(ii)] The degree map $CH_0(X_{k(X)})\to \Z$ is an isomorphism.
\item[(iii)] The class of $\Delta$ in
$CH_{\dim X}(X\times X)$ is equal to a class of the form
$\delta+\mathrm{pr}_2^*\omega$, where
$\delta$ is a cycle supported on $D\times X$
for some closed $D\subset X$ of codimension $1$
and $\omega$ is a $0$-cycle on $X$
of degree $1$.
\end{itemize}
Furthermore, the conditions imply the vanishing of the $n$-torsion in
$\Br(X)$ for every integer $n$ invertible in $k$.
\end{prop}

\begin{proof}
A proof can be found in \cite[Prop.~1.4]{colliotthelenepirutka} but we
sketch the argument for convenience.
The implications (i) $\Rightarrow$ (ii) $\Rightarrow$ (iii) are clear,
and (iii) $\Rightarrow$ (i) follows from
$(\mathrm{pr}_2)_*([\Delta]\cdot \mathrm{pr}_1^*\eta)=\eta$ for
$\eta\in CH_0(X)$ and a similar identity after base change to $L$.

Merkurjev has shown \cite[Thm.\ 2.11]{merkurjev}
that conditions (i)--(iii) are equivalent to $k(X)$ having only
trivial unramified elements for every cycle module $M$
in the sense of Rost \cite{rost}.
Applying this to \'etale cohomology with values in $\mu_n$
(cf.\ \cite{colliotthelenepurity}) yields
$\Br(X)[n]=0$ for $n$ invertible in $k$.
\end{proof}

For a proper algebraic variety $X$ over an
algebraically closed field $k$, we say that there is a
\emph{Chow decomposition of the diagonal} for $X$
if condition (iii) holds for $X$.
There is no loss of generality in taking $\omega$ in (iii) to be $\{x\}$ where
$x$ is a $k$-point of $X$.
The condition, to admit a Chow decomposition of the diagonal, also makes sense
for Chow groups with coefficients in an arbitrary commutative ring.

\begin{prop}
\label{prop.srat}
Let $k$ be an algebraically closed field, and
$X$ a smooth projective algebraic variety over $k$.
If $X$ is stably rational, then $X$ admits a Chow decomposition of the diagonal.
\end{prop}

\begin{proof}
If $X$ is stably rational, then $X$ is
retract rational \cite{saltman}: there exist an integer $N$,
nonempty open $U\subset X$ and $V\subset \A^N$,
and morphisms $\varphi\colon U\to V$ and $\psi\colon V\to U$
such that $\psi\circ \varphi=\mathrm{id}_U$.
By \cite[Lem.\ 1.5]{colliotthelenepirutka}, every smooth projective
variety that is retract rational admits
a Chow decomposition of the diagonal.
\end{proof}

\begin{prop}
\label{prop.desing}
Let $X$ be a proper three-dimensional algebraic variety over an
algebraically closed field $k$ with only ordinary double point singularities,
and $\widetilde{X}\to X$ the standard resolution obtained by
blowing up the singular points of $X$.
Then there is a Chow decomposition of the diagonal for $\widetilde{X}$ if and only if
there is a Chow decomposition of the diagonal for $X$.
\end{prop}

\begin{proof}
Since pushforward of cycles respects rational equivalence,
the forwards implication is clear.
The reverse implication follows from the fact that
each exceptional divisor of the
resolution is isomorphic to $\PP^1\times \PP^1$, and up to rational equivalence
any $3$-cycle in $\widetilde{X}\times \PP^1\times \PP^1$ may be written as a
cycle supported over a divisor on $\widetilde{X}$ and the
pullback of a $0$-cycle from $\PP^1\times \PP^1$.
\end{proof}

\section{Chow decomposition of the diagonal in families}
\label{sec.families}
The proof of the main theorem uses
the following version of Voisin's result.

\begin{theo}
\label{thm.voisin}
Let $B$ be an algebraic variety over an uncountable
algebraically closed field $k$
and $\pi\colon \mathcal{C}\to B$ a flat projective morphism with
integral fibers.
Let $R$ be the ring $\Z$ if $k$ has characteristic $0$, and
$\Z[1/p]$ if $k$ has characteristic $p>0$.
If there exists a $k$-point $b_0\in B$ such that the fiber
$\mathcal{C}_{b_0}$ does not admit a Chow decomposition of the diagonal
with coefficients in $R$,
then for very general $b\in B$ the fiber $\mathcal{C}_b$ does not admit
a Chow decomposition of the diagonal with coefficients in $R$.
\end{theo}

\begin{proof}
The set of $k$-points $b\in B$ such that the fiber
$\mathcal{C}_b$ admits a Chow decomposition of the diagonal with coefficients
in $R$
is a countable union of proper closed subvarieties of $B$.
In characteristic $0$ this is \cite[Thm.\ 2.3, App.~B]{colliotthelenepirutka}.
An initial reduction step in loc.\ cit.\ lets us assume that
$\pi$ admits a section $s\colon B\to \mathcal{C}$.

An argument valid in characteristic $p>0$ may be formulated using
the formalism of Chow functors of \cite{rydhthesis}.
We recall, rational equivalence of cycles may be formulated in terms of
cycles on a product with $\PP^1$; cf.\ \cite[\S 1.6]{fulton}.
We fix a sufficiently ample invertible sheaf $\mathcal{O}(1)$ on $\mathcal{C}$,
with which we measure degrees on the fibers of $\pi$;
for degrees of cycles on the fibers of $\mathcal{C}\times_B\mathcal{C}\to B$ and
$\mathcal{C}\times_B\mathcal{C}\times\PP^1\to B$ we use
$\mathcal{O}(1,1)$, respectively $\mathcal{O}(1,1,1)$.
Let $r$ denote the relative dimension of $\mathcal{C}$ over $B$.
For a positive integer $d$ we let $W_d$ be a scheme, quasi-projective over
the projective bundle $\PP(\pi_*\mathcal{O}(d))$ of relative degree $d$ divisors
$\mathcal{D}\subset \mathcal{C}$,
equipped with effective relative cycles
$\delta_1$ and $\delta_2$ on $\mathcal{D}\times_B\mathcal{C}$ of (relative) dimension $r$
and degree at most $dp^d$
and $\varepsilon$
on $\mathcal{C}\times_B\mathcal{C}\times \PP^1$ of dimension $r+1$
and degree at most $dp^d$, such that:
\begin{itemize}
\item[(i)] For every $w\in W_d$ and $z\in \PP^1_{k(w)}$ the cycle
$\varepsilon|_{\{w\}}$ meets
$\mathcal{C}\times_B\mathcal{C}\times \{z\}$ properly.
\item[(ii)] Every $k$-point $b\in B$, divisor
$D\in |\mathcal{O}_{\mathcal{C}_b}(d)|$, and triple of cycles satisfying (i)
occur at some $k$-point of $W_d$.
\end{itemize}
(Using Hilbert-to-Chow morphisms, the normalization
of a finite union of projective Hilbert schemes satisfies the
conditions, ignoring the proper intersection condition, and by
semicontinuity of fiber dimension the
proper intersection condition determines an open subscheme.)

For effective relative cycles there are operations of addition and
intersection with a relative Cartier divisor
\cite[(IV.4.15), (IV.15.2)]{rydhthesis}.
The equality of cycles
\begin{equation}
\label{eqn.equalityofcycles}
p^d\Delta+\delta_1+(\mathcal{C}\times_B\mathcal{C}\times \{0\})\cdot \varepsilon
=p^d(\mathcal{C}\times_Bs(B))+
\delta_2+(\mathcal{C}\times_B\mathcal{C}\times\{\infty\})\cdot \varepsilon
\end{equation}
defines a Zariski closed subset
\[ V_d\subset W_d. \]
Indeed, given a scheme $T$ of finite type over $B$
and pair of effective relative cycles of $\mathcal{C}\times_BT$,
proper and equidimensional of dimension $r$ over $T$,
the set of
points of $T$ at which the cycles are equal is Zariski closed in $T$.
Indeed, the support of a relative cycle over $T$
is universally open over $T$
\cite[(IV.4.7)]{rydhthesis}, and thus the locus
where the supports are equal is closed in $T$.
Given two cycles over $T$ with the same support $Z$,
in order to show that the locus of equality is closed it suffices
by Noetherian induction
to show: if every nonempty open $U\subset T$ has a point $u\in U$ at which
the cycles are equal, then the cycles are equal at every point of $T$.
Let us write $Z=Z_1\cup\dots\cup Z_\ell$ (irreducible components).
We may replace $T$ by its reduced subscheme, then we restrict our attention
to open $U$ contained in the normal locus of $T$, such that for every $i$
both $Z_i\times_TU$ and
$(Z_i\cap(Z_1\cup\dots\cup\widehat{Z}_i\cup\dots\cup Z_\ell))\times_TU$
are flat over $U$.
Then equality of cycles at a point $u\in U$ implies equality at the
generic point of the component of $U$ containing $u$
(cf.\ \cite[(IV.10.4)]{rydhthesis}).
Now, as in \cite[(IV.6.1)]{rydhthesis}, equality of cycles at the generic point of
every irreducible component of $T$ implies equality of cycles at
all points of $T$.

We let $U_d$ denote the image of $V_d$ in $B$.
We have $U_d\subset U_{d+1}$, and the $k$-points of the union
\[
\mathcal{U}:=\bigcup_{d=1}^{\infty} U_d
\]
are precisely the $k$-points $b\in B$ for which
$\mathcal{C}_b$ admits a Chow decomposition of the diagonal
with coefficients in $\Z[1/p]$.

We are done if we can show that for any positive integer $d$,
pointed smooth curve $(A,a)$, and morphism $A\to B$ such that the
image of the generic point of $A$ lies in $U_d$,
the image $b\in B$ of $a$ lies in $\mathcal{U}$.
It suffices to do this under the further assumption that the restriction
$A\smallsetminus\{a\}\to B$ factors through $V_d$.
Then there are corresponding effective cycles with $\Z[1/p]$-coefficients,
flat over $A\smallsetminus \{a\}$, and by
\cite[(IV.10.5)]{rydhthesis} these specialize at $k$-points of
$A\smallsetminus \{a\}$ to the respective corresponding cycles
$\delta_1$, $\delta_2$, $\varepsilon$.
Equation \eqref{eqn.equalityofcycles} translates into an equality of
cycles over $A\smallsetminus \{a\}$.
By applying the cycle-level specialization maps \cite[Rmk.\ 2.3]{fulton} and
the commutativity of pairs of specialization maps up to rational equivalence
\cite[Thm.\ 2.4]{fulton}, we obtain $b\in \mathcal{U}$.
\end{proof}

In combination with Proposition \ref{prop.desing} we therefore have:

\begin{coro}
\label{cor.voisin}
Let $B$ be an algebraic variety over an uncountable
algebraically closed field $k$
and $\mathcal{C}\to B$ a flat projective morphism of
relative dimension $3$ with integral fibers,
generically smooth with $k$-point $b_0\in B$ such that the fiber
$\mathcal{C}_{b_0}$ has only ordinary double point singularities.
Let $\widetilde{\mathcal{C}}_{b_0}\to \mathcal{C}_{b_0}$
denote the standard resolution obtained by
blowing up the singularities.
If $\Br(\widetilde{\mathcal{C}}_{b_0})[n]\ne 0$ for
some integer $n$ invertible in $k$, then for very general $b\in B$ the
fiber $\mathcal{C}_b$ is not stably rational.
\end{coro}

In particular, assuming the characteristic of $k$ to be different from $2$,
we may apply Corollary \ref{cor.voisin} in case
$\mathcal{C}_{b_0}$ is a conic bundle over a smooth projective
rational surface such that
\begin{itemize}
\item[(i)] The branch locus is a union $D_1\cup D_2$ where $D_1$ and $D_2$
are smooth curves meeting each other transversely.
\item[(ii)] The ramification data consists of degree $2$ \emph{unramified}
covers of $D_1$ and $D_2$.
\end{itemize}
Indeed, a conic bundle model of $\mathcal{C}_{b_0}$ with smooth total space
has branch locus the disjoint union of $D_1$ and $D_2$, and by \cite[Prop.~3]{AM}
(see also \cite{CTO})
there is nontrivial $2$-torsion in the unramified Brauer group of
$\mathcal{C}_{b_0}$.

\section{Construction}
\label{sec.constr}
\subsection{Construction I: Brauer class on root stacks}
\label{ss.constrI}
We let $S$ be a smooth rational projective surface over an
algebraically closed field $k$ of characteristic different from $2$.
Let $D\subset S$ be a divisor of the form $D_1\cup D_2$ where
$D_1$ and $D_2$ are smooth irreducible curves on $S$, each of
positive genus.
We assume that $D_1$ and $D_2$ meet transversely in some points
$p_1$, $\dots$, $p_r$.
Following Artin and Mumford \cite[Thm.~1]{AM}, letting $U:=S\smallsetminus D$,
a $2$-torsion element of the Brauer group
\[ \alpha\in \Br(U) \]
gives rise to double covers of $D_1$ and $D_2$ \'etale away from 
the intersection points; these determine $\alpha$ uniquely.
Moreover, \'etale double covers of $D_1$ and $D_2$ uniquely determine a
$2$-torsion element $\alpha \in \Br(U)$.

In the setting of $D=D_1\cup D_2$ on $S$ with given nontrivial \'etale double covers of
$D_1$ and $D_2$, the root stack
\[ X:=\sqrt{(S,D)} \]
has an ordinary double point singularity above each point $p_i$ and is
otherwise smooth.
Proposition \ref{prop.ramificationtimesn} applied to
$S\smallsetminus \{p_1,\dots,p_r\}$ shows that the class $\alpha$ is the
restriction of a unique element
\[ \beta\in \Br(X^{\mathrm{sm}}). \]

We will show that $\beta$ extends to a class in $\Br(X)$.
We start,
following Lieblich \cite{lieblichcrelle}, by producing an extension of
$\beta$ to a smooth compactification of $X^{\mathrm{sm}}$.
The smooth compactification that we use is
the \emph{iterated root stack}
\[ X':=\sqrt{(S,\{D_1,D_2\})} \]
introduced in \cite[Def.\ 2.2.4]{cadman}, which differs from the root stack $X$
in that the stabilizer group above each point $p_i$ is $\mu_2\times\mu_2$
rather than $\mu_2$.
The stack $X'$ admits a morphism
\[ \tau\colon X'\to X, \]
restricting to an isomorphism over $X^{\mathrm{sm}}$.
By Proposition \ref{prop.Brbasic}, there is a unique class
\[ \beta'\in \Br(X') \]
restricting to $\beta\in \Br(X^{\mathrm{sm}}).$

Given a $k$-point $p\in S$
we let $\mathcal{O}_{S,\bar p}$ denote the
(strict) Henselization of the local ring of $S$ at $p$.

\begin{lemm}
\label{lem.etalelocallytrivial}
For every $j$ the class $\alpha$ lies in the kernel of
\[ \Br(U)\to \Br(\Spec(\mathcal{O}_{S,\bar p_j})\times_SU). \]
\end{lemm}

\begin{proof}
This is immediate from Proposition \ref{prop.Brbasic} and the fact that
the degree $2$
covers of $D_1$ and $D_2$ are \'etale.
\end{proof}

\begin{prop}
\label{prop.azumaya1}
The element $\alpha$ is represented by a sheaf of Azumaya algebras
$\mathcal{A}'$ on $X'$ of index $2$.
\end{prop}

\begin{proof}
It is a general fact for $C_2$-fields of characteristic different from $2$,
e.g., the function field $k(S)$, that
a nontrivial $2$-torsion Brauer group element is
the class of a quaternion algebra; see \cite[\S 5.7]{sarkisovcb}.
Since $X'$ is Noetherian and regular of dimension $2$, such a quaternion algebra
is the restriction of a sheaf of Azumaya algebras.
\end{proof}

Let $f\colon Z\to W$ be a
finite-type morphism of
separated Noetherian Deligne-Mumford stacks.
Then there is a relative notion of corase moduli space,
introduced in \cite{AOVtwisted}, which determines a
stack with representable morphism to $W$, through which $f$ factors.
The relative moduli space is characterized by the standard sort of
universal property, after base change to an \'etale atlas of $W$.
In many cases $W$ already has this property; then we
will say, $f$ is a \emph{relative moduli space}.
Root stacks are one important class of examples.
The morphism $\tau$ above is also a relative moduli space.

Under a tameness hypothesis, satisfied by our assumption that the
characteristic of $k$ is different from $2$, a locally free coherent sheaf is
obtained as
pullback of a locally free coherent sheaf from a relative moduli space
if and only if
the relative stabilizer group actions on fibers at closed points
are trivial \cite[Thm.\ 10.3]{alper}.

\begin{prop}
\label{prop.stabilizeraction}
The action of the diagonal $\mu_2$ in the stabilizer $\mu_2\times\mu_2$
of the closed point
of $\Spec(\mathcal{O}_{S,\overline{p}_j})\times_S X'$
on the fiber of $\mathcal{A}'$ is trivial.
\end{prop}

\begin{proof}
By Lemma \ref{lem.etalelocallytrivial},
for each $j$ the pullback of $\beta'$ to
$\Spec(\mathcal{O}_{S,\overline{p}_j})\times_SX'$ vanishes.
So the $PGL_2$-representation coming from $\mathcal{A}'$
of the stabilizer $\mu_2\times \mu_2$ at the
closed point of $\Spec(\mathcal{O}_{S,\overline{p}_j})\times_SX'$ must
lift to a $GL_2$-representation, which is well-defined up to twist by a
character of $\mu_2\times \mu_2$.
Thus we may take it
to have the form of the direct sum of a trivial representation
and a character.
We claim that the character must be the composite
\[ \mu_2\times \mu_2\to \mu_2\to k^\times \]
of the multiplication and the nontrivial character of $\mu_2$.
Indeed, the criterion for a locally free coherent sheaf to be
a pullback via
\[ X'\to \sqrt{(S,D_i)}, \]
for $i\in \{1,2\}$, is nontrivial only on the gerbe of the root stack
over $D_{3-i}$, and since $D_{3-i}$ is connected the criterion is satisfied
everywhere if it is satisfied at one point.
So the claim follows from the fact that $\alpha$ does not extend to the
generic point of $D_{3-i}$.
\end{proof}

\begin{coro}
\label{cor.azumayaX}
The sheaf of Azumaya algebras $\mathcal{A}'$ is isomorphic to the
pullback of a
sheaf of Azumaya algebras $\mathcal{A}$ on $X$.
\end{coro}

\begin{proof}
This follows by the criterion for pullback of a locally free coherent sheaf,
applied to $\tau\colon X'\to X$.
\end{proof}

The class of $\mathcal{A}$ in $\Br(X)$ extends
$\beta\in \Br(X^{\mathrm{sm}})$ and will be denoted as well
by $\beta$.

\subsection{Construction II: Conic bundle}
\label{ss.constrII}
We have $G=[P/SL_2]$, the $\mu_2$-gerbe
of Proposition \ref{prop.azumayagerbe}
applied to the stack $X$ and the sheaf of Azumaya algebras $\mathcal{A}$,
which, we may recall, is classified by the cohomology class
\[ \gamma\in H^2(X,\mu_2) \]
coming from $\mathcal{A}$ via the boundary homomorphism of
nonabelian cohomology.

\begin{lemm}
\label{lem.Etilde}
Suppose that $\widetilde{E}$ is a locally free sheaf of rank $2$
on the $\mu_2$-gerbe $G$ with nontrivial generic stabilizer action.
Then the sheaf of Azumaya algebras $\widetilde{\mathcal{A}}$ on $X$,
defined uniquely up to
isomorphism by the requirement that the pullback of $\widetilde{\mathcal{A}}$
to $G$ is isomorphic to $\widetilde{E}^\vee\otimes \widetilde{E}$,
is Brauer equivalent to $\mathcal{A}$.
\end{lemm}

\begin{proof}
The representable morphism
\[ [P/SL_2]\to [P/GL_2] \]
from the $\mu_2$-gerbe to the $\G_m$-gerbe identifies the categories of
locally free coherent sheaves for which the generic stabilizer action is
the scalar action by the inclusion $\mu_2\to \G_m$, respectively the
identity morphism of $\G_m$; cf.\ \cite[Lemma 3.5]{lieblichstjohn}.
So there is a locally free sheaf $\widetilde{F}$ of rank $2$ on
$[P/GL_2]$ whose pullback to $G$ is isomorphic to $\widetilde{E}$.
The associated principal $GL_2$-bundle of $\widetilde{F}$ may be identified
with the principal $PGL_2$-bundle $\widetilde{P}$ associated with
$\widetilde{\mathcal{A}}$,
inducing
\[ [P/GL_2] \cong [\widetilde{P}/GL_2] \]
and hence an equality of the corresponding Brauer classes.
\end{proof}

\begin{prop}
\label{prop.conicbundle}
With $S$, $D$, $X=\sqrt{(S,D)}$, $\gamma\in H^2(X,\mu_2)$, and $G$ as above,
let $(B,b_0)$ be a smooth pointed
algebraic variety over $k$ and $\mathcal{D}\subset B\times S$ an
irreducible smooth divisor such that
$\mathcal{D}\cap (\{b_0\}\times S)=D$.
We let $\mathcal{X}$ denote the root stack $\sqrt{(B\times S,\mathcal{D})}$.
If
\[ \Gamma\in H^2(\mathcal{X},\mu_2) \]
is a class restricting to $\gamma$ under the inclusion of $X$ in
$\mathcal{X}$, and we let
\[ \mathcal{G}\to \mathcal{X} \]
denote a $\mu_2$-gerbe with class $\Gamma$, then for any
locally free sheaf $\widetilde{\mathcal{E}}$ of rank $2$ on $\mathcal{G}$ with
nontrivial generic stabilizer action, the
associated smooth conic fibration
\[ \widetilde{\mathcal{C}}\to \mathcal{X}, \]
characterized uniquely up to isomorphism by
\[ \mathcal{G}\times_{\mathcal{X}}
\widetilde{\mathcal{C}}\cong \PP(\widetilde{\mathcal{E}}), \]
has nontrivial stack structure at precisely two geometric points above every
geometric point of $\mathcal{D}$.
The blow-up
\[ B\ell_{\widetilde{\mathcal{C}}^{\mathrm{stack}}}\widetilde{\mathcal{C}}\to
\mathcal{X} \]
of the locus with nontrivial stack structure
is a flat family of genus $0$ prestable curves whose geometric fibers
are irreducible over points of $\mathcal{X}\smallsetminus \mathcal{D}$ and
chains of three irreducible components over points of
$\mathcal{D}$.
Collapsing the ``middle component'' of the three-component chains yields a
flat family of genus $0$ prestable curves which is isomorphic to the
base-change by $\mathcal{X}\to B\times S$ of a standard conic bundle
\[ \mathcal{C}\to B\times S. \]
\end{prop}

\begin{proof}
Since $\mathcal{D}$ is connected, it suffices to check the structure of
the locus with stack structure above a single point.
By hypothesis, $\mathcal{G}\times_{\mathcal{X}}X$ may be identified with $G$.
By Lemma \ref{lem.Etilde} applied to $\widetilde{\mathcal{E}}|_G$,
we may apply the analysis of Proposition \ref{prop.stabilizeraction}
to the sheaf of Azumaya algebras associated with
$\widetilde{\mathcal{E}}|_G$
(cf.\ Lemma \ref{lem.Etilde}), pulled back to $X'$, to see that the
stack structure is as claimed.
The blow-up is smooth, so flatness follows from the fibers being of
constant dimension, and the geometric fibers are as claimed.
The collapsing of the middle components is a standard construction,
achievable for instance by formation of Proj of the sum of the direct images
of powers of the relative anticanonical sheaf; cf.\ \cite[\S 5.2]{kresch}.
Since $\mathcal{X}\to B\times S$ is flat (a general property of
root stacks), descent to $B\times S$
follows from the observation that the stabilizer action on singular fibers
becomes trivial after collapsing the middle components.
\end{proof}

\subsection{Construction III: Elementary transform}
\label{ss.constrIII}
We will obtain a locally free sheaf $\widetilde{\mathcal{E}}$ as in
Proposition \ref{prop.conicbundle} using deformation theory.
The strategy is to start with the locally free sheaf $E$ on $G$ of
Proposition \ref{prop.azumayagerbe} and apply a standard modification,
called elementary transform, to produce a
locally free sheaf $\widetilde{E}$ of rank $2$ on $G$, which satisfies a
deformation-theoretic condition to extend to a
locally free sheaf $\widetilde{\mathcal{E}}$.

The locally free sheaf $E$ of Proposition \ref{prop.azumayagerbe} has
the property that $E^\vee\otimes E$ descends to the
sheaf of Azumaya algebras $\mathcal{A}$ on $X$.
In this and similar situations we will also use the notation $E^\vee\otimes E$
for the corresponding locally free sheaf on $X$.

The deformation theory of
locally free coherent sheaves with given determinant is recalled, e.g.,
in \cite{ogrady}.
The space of obstructions for $E$ is the group
\[ \ker\big(H^2(X, E^\vee\otimes E)\to H^2(X,\mathcal{O}_X)\big), \]
kernel of the trace homomorphism.
We may work on $X'$, since $X$ is a relative coarse moduli space
of $X'$, hence $R^i\tau_*\mathcal{F}'=0$ for every
quasicoherent sheaf $\mathcal{F}'$ on $X'$ and $i>0$
(cf.\ \cite[\S 3]{AOVtame}).
There is a gerbe $G'$, fitting into a fiber diagram:
\[
\xymatrix{
G' \ar[r] \ar[d]_\sigma & X' \ar[d]^\tau \\
G \ar[r] & X
}
\]
Letting
\[ E':=\sigma^*E, \]
we have
\[ H^i(X,E^\vee\otimes E)\cong H^i(X',E'^\vee\otimes E'). \]

Since $X'$ is smooth, tame, with projective coarse moduli space,
there is Serre duality with the standard dualizing sheaf $K=K_{X'}$
(\cite[\S 2]{nironi}, \cite[App.\ B]{bruzzosala}).
Furthermore, since the rank $2$ of $E$ is invertible in $k$, the trace
homomorphism admits a canonical splitting.
Combining these facts, we identify the space of obstructions with the dual of
\[ H^0(X',(E'^\vee\otimes E')_0\otimes K):=
H^0(X',\ker(E'^\vee\otimes E'\otimes K\to K)). \]
The kernel, here and below, is denoted with subscript $0$.

Let $C$ be a smooth irreducible orbifold curve on $X$, contained
entirely in $X^{\mathrm{sm}}$, let $E|_C$ denote by abuse of notation
the restriction of $E$ to $G\times_XC$, and let
\begin{equation}
\label{eqn.REQ}
0\to R\to E|_C\to Q\to 0
\end{equation}
be an exact sequence of coherent sheaves on $G\times_XC$,
where $Q$ and $R$ are locally free of rank $1$.
Then we define $\widetilde{E}$ to be the kernel of the composite of
restriction to $G\times_XC$ and the map to $Q$ from \eqref{eqn.REQ}, so that
we have the exact sequence
\[ 0\to \widetilde{E}\to E\to Q\to 0. \]

\begin{prop}
\label{prop.elementarytransform}
There exist such an orbifold curve $C$ and such an exact sequence
\eqref{eqn.REQ} of locally free coherent sheaves on $G\times_XC$, such that
the space of obstructions for the elementary transform $\widetilde{E}$
vanishes.
\end{prop}

\begin{proof}
We may work
on $X'$ with
\[ \widetilde{E}':=\sigma^*\widetilde{E} \]
and show that
\[ H^0(X',(\widetilde{E}'^\vee\otimes \widetilde{E}')_0\otimes K)=0. \]
The argument follows closely that of \cite{dejong}
(written in the language of sheaves of Azumaya algebras rather than
vector bundles on gerbes).

There is the exact sequence
\[
0\to E'^\vee\to \widetilde{E}'^\vee\to Q^\vee(C) \to 0
\]
obtained by identifying
$\mathrm{Ext}^1(Q,\mathcal{O}_{G'})$ with $Q^\vee(C)$.
There are standard exact sequences
involving the intersection $F'$ of
\[ E'^\vee\otimes E'\qquad
\text{and}\qquad
\widetilde{E}'^\vee\otimes \widetilde{E}' \]
in $\widetilde{E}'^\vee\otimes E'$:
\begin{gather*}
0\to F'\to E'^\vee\otimes E'\to R^\vee\otimes Q\to 0,\\
0\to F'\to \widetilde{E}'^\vee\otimes \widetilde{E}'\to
Q^\vee(C)\otimes R\to 0,
\end{gather*}
giving rise to similar exact sequences involving the kernels
of the trace map
\[ F'_0,\qquad (E'^\vee\otimes E')_0,\qquad (\widetilde{E}'^\vee\otimes \widetilde{E}')_0. \]
Now it suffices to show:
\begin{itemize}
\item[(i)] There exist a finite set of points $\{x_1,\dots,x_n\}$ of
$U$ and points
\[ r_1\in \PP(E_{x_1}),\qquad\dots,\qquad r_n\in \PP(E_{x_n}) \]
such that for $C$ passing through $x_1$, $\dots$, $x_n$ and exact sequence
\eqref{eqn.REQ} such that $[R_{x_i}]=r_i$ for all $i$, the map on
global sections of
\[ (E'^\vee\otimes E')_0\otimes K\to R^\vee\otimes Q\otimes K \]
is injective.
\item[(ii)] With points $x_1$, $\dots$, $x_n$ and $r_1$, $\dots$, $r_n$
as in (i) and any choice of $C$ passing through $x_1$, $\dots$, $x_n$,
there exists an exact sequence \eqref{eqn.REQ} with
$[R_{x_i}]=r_i$ for all $i$ and $R$ of arbitrarily negative degree.
\end{itemize}

Notice, then, $H^0(X',F'_0\otimes K)=0$ and
$H^0(C,Q^\vee(C)\otimes R\otimes K)=0$, and we conclude that
$H^0(X',(\widetilde{E}'^\vee\otimes \widetilde{E}')_0\otimes K)$ vanishes,
as desired.

For (i), we may fix a trivialization of $K$ on some nonempty open
$U_1\subset U$.
For nonzero $s\in H^0(X',(E'^\vee\otimes E')_0\otimes K)$ and
$x\in U_1$ such that $s$ is nonzero in the fiber over $x$,
we identify $s$ (using the trivialization of $K$) with a trace zero
endomorphism of $E'_x$.
For any one-dimensional subspace $R_x\subset E'_x$ which is not an eigenspace,
$s$ maps to a nonzero element of $\Hom(R_x,E'_x/R_x)$.
Since $H^0(X',(E'^\vee\otimes E')_0\otimes K)$ is finite dimensional,
some finite collection of points and subspaces will satisfy (i).

For (ii) we apply Tsen's theorem, which gives the vanishing of the
restriction of $\beta$ to $C$.
So by Proposition \ref{prop.azumayatrivial} there is a line bundle $T$
on $G\times_XC$ such that the generic stabilizer action is nontrivial.
Then $E\otimes T^\vee$ descends to a vector bundle
\[ N \to C,\]
and we obtain a
sequence \eqref{eqn.REQ} from a similar sequence involving
the vector bundle $N$ by tensoring with $T$.
Notice, to give such a sequence is the same as to give a section of
\begin{equation}
\label{eqn.PNC}
\PP(N)\to C,
\end{equation}
which, being a representable projective morphism, is the same as to give a
section of the restriction of \eqref{eqn.PNC} to a nonempty open substack of $C$.
We choose $V\subset C$ that contains all the points $x_1$, $\dots$, $x_n$,
avoids all the orbifold points, and has the property that $N|_V$ is trivial.
A choice of trivialization identifies sections of \eqref{eqn.PNC} with
maps $C\to \PP^1$, which may be chosen to map $x_i$ to the point corresponding
to $r_i$ for every $i$ and to have arbitrarily high degree.
\end{proof}

\section{Proof of the main theorem}
Let the notation be as in Theorem \ref{thm.main}.
By hypothesis,
$\mathcal{M}$ contains a point
corresponding to an \'etale degree $2$ cover,
nontrivial over every irreducible component of a reducible
nodal curve with smooth irreducible compoments.
We take, initially, this as the pointed variety $(B,b_0)$.
For simplicity of notation we write $D=D_1\cup D_2$; the case
of more than $2$ irreducible components creates no additional difficulty.

The covering of $D$
determines a nontrivial $2$-torsion
Brauer group element $\alpha\in \Br(S\smallsetminus D)$.
By Corollary \ref{cor.azumayaX}, the class $\alpha$ pulls back and extends to
the class $\beta\in \Br(X)$ of a sheaf of
Azumaya algebras $\mathcal{A}$ of index $2$ on the root stack $X=\sqrt{(S,D)}$.
Now we fix a $\mu_2$-gerbe $G\to X$ and locally free coherent sheaf $E$ on $G$
as in Proposition \ref{prop.azumayagerbe} and let $\gamma\in H^2(X,\mu_2)$
be the class of $G$.

By Proposition \ref{prop.elementarytransform}, there is an
elementary transform $\widetilde{E}$ of $E$ with
vanishing obstruction for the deformation theory of
locally free coherent sheaf with given determinant.

We let $\mathcal{D}\subset B\times S$ be the divisor corresponding to the
linear system underlying $B$ and let
$\mathcal{X}$ denote the root stack $\sqrt{(B\times S,\mathcal{D})}$.
By the proper base change theorem for tame Deligne-Mumford stacks
\cite[App.\ A]{ACV} applied to $\mathcal{X}\to B$, we may
replace $B$ by an \'etale neighborhood of $b_0$ in such a way that
the class $\gamma$ is the restriction of a class
\[ \Gamma\in H^2(\mathcal{X},\mu_2). \]
We denote an associated $\mu_2$-gerbe by
\[ \mathcal{G}\to \mathcal{X}. \]

\begin{lemm}
\label{lem.detEtilde}
The restriction map $\Pic(\mathcal{X})\to \Pic(X)$ is surjective.
\end{lemm}

\begin{proof}
By the description of the Picard group of a root stack
\cite[\S 3.1]{cadman}, for a given line bundle $L$ on $X$ either
$L$ or its twist by the divisor of the gerbe of root stack is
pulled back from $X$.
Since every line bundle on $X=\{b_0\}\times X$ is the restriction of a
line bundle on $B\times X$, we have the result.
\end{proof}

We apply Lemma \ref{lem.detEtilde} to the line bundle on $X$ whose
pullback to $G$ is equal to $\det(\widetilde{E})$ to obtain a line
bundle on $\mathcal{X}$ whose pullback to $\mathcal{G}$ restricts to
$\det(\widetilde{E})$ on $G$.
Now the deformation theory of locally free coherent sheaves with
given determinant (which in our setting makes use of the
Grothendieck existence theorem for tame Deligne-Mumford stacks
of \cite[App.\ A]{AV}), gives, after
replacing $B$ by an \'etale neighborhood of $b_0$, the existence of
a locally free coherent sheaf $\widetilde{\mathcal{E}}$ on $\mathcal{G}$
extending $\widetilde{E}$.

Proposition \ref{prop.conicbundle}, applied to $\widetilde{\mathcal{E}}$,
yields a standard conic bundle $\mathcal{C}\to B\times S$
whose restriction over $\{b_0\}\times S$ has
an ordinary double point singularity over every point of
$D_1\cap D_2$.
We conclude by applying Corollary \ref{cor.voisin} to the family
$\mathcal{C}\to B$.


\begin{thebibliography}{19}

\bibitem{ACV} D. Abramovich, A. Corti, and A. Vistoli,
\emph{Twisted bundles and admissible covers},
Comm. Algebra 31 (2003), 3547--3618.

\bibitem{AGV} D. Abramovich, T. Graber, and A. Vistoli,
\emph{Gromov-Witten theory of Deligne-Mumford stacks},
Amer. J. Math. 130 (2008), 1337--1398.

\bibitem{AOVtame} D. Abramovich, M. Olsson, and A. Vistoli,
\emph{Tame stacks in positive characteristic},
Ann. Inst. Fourier (Grenoble) 58 (2008), 1057--1091.

\bibitem{AOVtwisted} D. Abramovich, M. Olsson, and A. Vistoli,
\emph{Twisted stable maps to tame Artin stacks},
J. Algebraic Geom. 20 (2011), 399--477.

\bibitem{AV} D. Abramovich and A. Vistoli,
\emph{Compactifying the space of stable maps},
J. Amer. Math. Soc. 15 (2002), 27--75.

\bibitem{alper} J. Alper,
\emph{Good moduli spaces for Artin stacks},
Ann. Inst. Fourier (Grenoble) 63 (2013), 2349--2402.

\bibitem{AM} M. Artin and D. Mumford,
\emph{Some elementary examples of unirational varieties which are not rational},
Proc. London Math. Soc. (3) 25 (1972), 75--95.

\bibitem{ACTP} 
A. Auel, J.-L. Colliot-Th\'el\`ene, and R. Parimala,
\emph{Universal unramified cohomology of cubic fourfolds containing a plane},
to appear in
Brauer groups and obstruction problems: moduli spaces
and arithmetic (Palo Alto, CA, 2013).

\bibitem{beauvilleprym} A. Beauville,
\emph{Vari\'et\'es de Prym et jacobiennes interm\'ediaires},
Ann. Sci. \'Ecole Norm. Sup. (4) 10 (1977), 309--391.

\bibitem{Beau} A. Beauville,
\emph{Le groupe de monodromie des familles universelles
d'hypersurfaces et d'intersections compl\`etes},
Complex analysis and algebraic geometry (G\"ottingen, 1985),
Lecture Notes in Math. 1194, Springer, Berlin, 1986, pp. 8--18.

\bibitem{beauvillefourfive} A. Beauville,
\emph{A very general quartic double fourfold or fivefold is not stably rational},
arXiv:1411.3122.

\bibitem{beauvillesextic} A. Beauville,
\emph{A very general sextic double solid is not stably rational},
arXiv:1411.7484.

\bibitem{bruzzosala} U. Bruzzo and F. Sala,
\emph{Framed sheaves on projective stacks},
Adv. Math. 272 (2015), 20--95.

\bibitem{cadman} C. Cadman,
\emph{Using stacks to impose tangency conditions on curves},
Amer. J. Math. 129 (2007), 405--427.

\bibitem{colliotthelenepurity}
J.-L. Colliot-Th\'el\`ene,
\emph{Birational invariants, purity and the Gersten conjecture},
in K-theory and algebraic geometry:
connections with quadratic forms and division algebras (Santa Barbara, CA, 1992),
Proc. Sympos. Pure Math., 58, Part 1,
Amer. Math. Soc., Providence, RI, 1995, pp. 1--64.

\bibitem{CTO}
J.-L. Colliot-Th\'el\`ene and M. Ojanguren,
\emph{Vari\'et\'es unirationnelles non rationnelles: au-del\`a de l'exemple d'Artin et Mumford},
Invent. Math. 97 (1989), no. 1, 141--158.

\bibitem{colliotthelenepirutka}
J.-L. Colliot-Th\'el\`ene and A. Pirutka,
\emph{Hypersurfaces quartiques de dimension 3: non rationalit\'e stable},
Ann. Sci. \'Ecole Norm. Sup. (4), to appear.

\bibitem{dejong} A. J. de Jong,
\emph{The period-index problem for the Brauer group of an algebraic surface},
Duke Math. J. 123 (2004), 71--94.

\bibitem{fulton} W. Fulton,
\emph{Intersection Theory}, 2nd ed.,
Springer, Berlin, 1998.

\bibitem{GB}
A. Grothendieck,
\emph{Le groupe de Brauer, I--III},
in Dix expos\'es sur la cohomologie des sch\'emas,
Advanced Studies in Pure Math. 3, North-Holland, Amsterdam, 1968,
pp. 46--188.

\bibitem{HKT}
B. Hassett, A. Kresch, and Y. Tschinkel,
\emph{On the moduli of degree 4 Del Pezzo surfaces},
to appear in Development of moduli theory (Kyoto, 2013).

\bibitem{iskovskikh}
V. A. Iskovskikh,
\emph{On the rationality problem for conic bundles},
Duke Math. J. 54 (1987), 271--294.


\bibitem{kollar}
J. Koll\'ar,
\emph{Nonrational hypersurfaces}, 
J. Amer. Math. Soc. 8 (1995), 241--249.


\bibitem{kresch}
A. Kresch,
\emph{Flattening stratification and the stack of partial stabilizations of
prestable curves},
Bull. London Math. Soc. 45 (2013), 93--102.

\bibitem{LMB} G. Laumon and L. Moret-Bailly,
\emph{Champs Alg\'ebriques},
Springer, Berlin, 2000.

\bibitem{lieblichcrelle} M. Lieblich,
\emph{Period and index in the Brauer group of an arithmetic surface},
J. Reine Angew. Math. 659 (2011), 1--41.

\bibitem{lieblichstjohn} M. Lieblich,
\emph{On the ubiquity of twisted sheaves},
in Birational geometry, rational curves, and arithmetic,
Springer, New York, 2013, pp. 205--227.

\bibitem{mella} M. Mella,
\emph{On the unirationality of $3$-fold conic bundles},
arXiv:1403.7055.

\bibitem{merkurjev} A. Merkurjev,
\emph{Unramified elements in cycle modules},
J. London Math. Soc. (2) 78 (2008), 51--64.

\bibitem{milne} J. S. Milne,
\emph{\'Etale Cohomology},
Princeton Univ. Press, Princeton, NJ, 1980.

\bibitem{nironi} F. Nironi,
\emph{Grothendieck duality for Deligne-Mumford stacks},
arXiv:0811.1955.

\bibitem{ogrady} K. O'Grady,
\emph{Moduli of vector bundles on surfaces},
in Algebraic Geometry (Santa Cruz, 1995),
Proc. Symposia Pure Math. 62, Part 1,
Amer. Math. Soc., Providence, RI, 1997, pp. 101--126.

\bibitem{rost} M. Rost,
\emph{Chow groups with coefficients},
Doc. Math. 1 (1996), 319--393.

\bibitem{rydhthesis} D. Rydh,
\emph{Families of cycles and the Chow scheme},
Ph. D. thesis, KTH, Stockholm, 2008.

\bibitem{rydhjalg} D. Rydh,
\emph{\'Etale d\'evissage, descent and pushouts of stacks},
J. Algebra 331 (2011), 194--223.

\bibitem{saltman} D. J. Saltman,
\emph{Retract rational fields and cyclic Galois extensions},
Israel Math. J. 47 (1984), 165--215.

\bibitem{sarkisovbir} V. G. Sarkisov,
\emph{Birational automorphisms of conic bundles},
Izv. Akad. Nauk SSSR Ser. Mat. 44 (1980), 918--945.

\bibitem{sarkisovcb} V. G. Sarkisov,
\emph{On conic bundle structures},
Izv. Akad. Nauk SSSR Ser. Mat. 46 (1982), 371--408.

\bibitem{shokurov} 
V. V. Shokurov,
\emph{Prym varieties: theory and applications},
Izv. Akad. Nauk SSSR Ser. Mat. 47 (1983), 785--855.

\bibitem{totaro} B. Totaro,
\emph{Hypersurfaces that are not stably rational},
arXiv:1502.04040.

\bibitem{voisin} C. Voisin,
\emph{Unirational threefolds with no universal codimension $2$ cycle},
Invent. Math., to appear.

\end{thebibliography}
\end{document}